\documentclass[a4paper,12pt]{amsart}
\usepackage[utf8]{inputenc}
\usepackage[dvips]{graphicx}
\usepackage[shortlabels]{enumitem}
\usepackage{amsmath,amssymb,color,enumitem,graphics,hyperref,latexsym,vmargin,yfonts}
\usepackage{mathtools}
\usepackage[lf]{Baskervaldx} 

\mathtoolsset{showonlyrefs}

\hypersetup{
    unicode=false,
    pdftoolbar=true,
    pdfmenubar=true,
    pdffitwindow=false,
    pdfstartview={FitH},
    pdfnewwindow=true,
    colorlinks=true,
    linkcolor=black,
    citecolor=black,
    filecolor=magenta,
    urlcolor=cyan
}

\def\cB{{\mathcal{B}}}
\def\C{{\mathbb{C}}}
\def\D{{\mathbb{D}}}
\def\cH{{\mathcal{H}}}
\def\T{{\mathbb{T}}}
\def\Z{{\mathbb{Z}}}
\def\N{{\mathbb{N}}}
\def\bP{{\mathbf{P}}}

\DeclareRobustCommand{\rchi}{{\mathpalette\irchi\relax}}
\newcommand{\irchi}[2]{\raisebox{\depth}{$#1\chi$}} 

\newcommand{\Hardy}{H^2(\mathbb{T}^N)}
\newcommand{\norm}[1]{\left\lVert#1\right\rVert}

\newtheorem{theorem}{Theorem}[section]
\newtheorem{lemma}[theorem]{Lemma}
\newtheorem{proposition}[theorem]{Proposition}

\theoremstyle{definition}

\newtheorem{remark}[theorem]{Remark}
\newtheorem{example}[theorem]{Example}

\begin{document}

\title{Frames of orbits of multiplication operators on Hardy spaces}

\author[A. Aguilera, D. Carando]{Alejandra Aguilera \and Daniel Carando}
\address{Departamento de Matem\'atica, Facultad de Ciencias Exactas y Naturales, Universidad de Buenos Aires and Instituto de Matem\'atica ``Luis Santal\'o'' (CONICET-UBA), Buenos Aires, Argentina.}
\email{aaguilera@dm.uba.ar}
\email{dcarando@dm.uba.ar}

\let\thefootnote\relax

\begin{abstract}

We study frames for Hardy spaces generated by orbits of multiplication operators.  We characterize the symbols $\varphi \in H^\infty(\mathbb{T}^N)$ for which the multiplication operator $M_\varphi$ admits a frame of orbits  on $H^2(\mathbb{T}^N)$.
 We also show that, in this 
setting, the existence of a frame is equivalent to the existence of a 
Parseval frame.  Moreover, for $N=1$ we prove that finitely many 
orbits suffice if and only if $\varphi$ is a finite Blaschke product. For $N > 1$, no finite 
collection of orbits can generate a frame, regardless of the symbol. We study the analogous problem for the adjoint operator $M_\varphi^*$. Our results extend to the infinite-dimensional torus $\mathbb{T}^\infty$ 
and, via Bohr's transform, to the Hardy space of Dirichlet series 
$\mathcal{H}_2$.
\end{abstract}

\maketitle

\section*{Introduction}

{Motivated by applications in dynamical sampling, signal processing, and the theory of shift-invariant systems, recent research has focused on frames generated by the iterative action of a bounded operator on a Hilbert space (see e.g. \cite{ACCMP17,ADK13,ADK15,ACMT,AP17}).}

Given a bounded linear operator 
$T$ on a separable Hilbert space $\mathcal{H}$ and a countable collection of vectors 
$\{f_i\}_{i \in I} \subseteq \mathcal{H}$, one seeks conditions under which 
the system of \emph{orbits}
$$\{T^n f_i\}_{i \in I,n \geq 0}$$
constitutes a frame for $\mathcal{H}$ (see, for example,  \cite{ ACNP25, AP17, CMPP}).

In this paper, we study the case where $T$ is a multiplication 
operator $M_\varphi$ (or its adjoint $M_\varphi^*$) acting on Hardy spaces. 
Specifically, given $N \in \mathbb{N}$ and a symbol 
$\varphi \in H^\infty(\mathbb{T}^N)$, we study when the system 
$\{M_\varphi^n f_i\}_{i \in I, n \geq 0}$ forms a frame for 
$H^2(\mathbb{T}^N)$. Subspaces of Hardy spaces
on the one-dimensional torus proved very useful as models for general Hilbert spaces in the study of dynamical frames \cite{ACCP23,CMS21,CMS,CHP}. This fact motivated us to investigate the dynamical frame problem directly within these spaces. Moreover, multiplication operators are  fundamental operators in function spaces and, in particular, in Hardy spaces. 

We remark that, in \cite{KasShu19}, the authors show that we cannot expect to have an orbit of the multiplication operator to be a frame of  $L^2(a,b)$. The situation is different in $H^2(\T)$, since   $\{z^n\}_{n\ge 0} $  is an orthonormal basis, but similar in $H^2(\T^N)$ for $N>1$. Our results, in particular, answer the main question posed in \cite{Ches}.

Our first main result (Theorem~\ref{frame_varphi<=1}) characterizes those symbols $\varphi \in H^\infty(\mathbb{T}^N)$ for 
which $M_\varphi$ admits a frame of orbits (with possibly infinitely many 
generators): this occurs if and only if $\|\varphi\|_\infty \leq 1$ and 
 $|\widehat{\varphi}(0)| < 1$. It is interesting to note that, in this context, admiting Parseval 
frames or general frames turn out to be equivalent. 

Another natural question concerns the { number} of orbits 
required to generate a frame. Here, we find a dichotomy between the one-dimensional and the multidimensional settings. For $N = 1$, we prove 
(Theorem~\ref{teo-frames-N=1}) that we can have finitely many orbits if and 
only if $\varphi$ is a finite Blaschke product; in this case the operator even admits an orthonormal basis which can be obtained by taking  the Malmquist--Walsh basis as the set of generators. In  contrast, for $N > 1$ 
 no finite collection of orbits can 
generate a frame, regardless of the symbol (Proposition~\ref{prop-no-finitos}). 

We also address the question on the adjoint operator $M_\varphi^*$ admiting frames of orbits. In 
Theorem~\ref{teo-phi1}, we show that $M_\varphi^*$ admits a frame of 
orbits if and only if $|\varphi(\zeta)| < 1$ almost everywhere on 
$\mathbb{T}^N$, and that in this case the index set $I$ must necessarily 
be infinite. For adjoint operators there is no difference between one-dimensional and multidimensional cases.

Finally, in Section~4 we extend our results to the Hardy space of Dirichlet 
series $\mathcal{H}_2$ via Bohr's transform, which provides an isometric 
isomorphism between the  space $\mathcal{H}_2$ of Dirichet series and $H^2(\mathbb{T}^\infty)$. The characterization of those Dirichlet series 
$D_0 \in \mathcal{H}_\infty$ for which the multiplication operator 
$M_{D_0}$ on $\mathcal{H}_2$ admits a frame of orbits is given in 
Theorem~\ref{frame_D_0}.

The paper is organized as follows. Section~2 recalls background on frames 
in Hilbert spaces, Hardy spaces on the polydisc and torus, and reproducing 
kernel Hilbert spaces. Section~3 contains our main results for $M_\varphi$ 
and $M_\varphi^*$ on $H^2(\mathbb{T}^N)$. Section~4 presents the Dirichlet 
series setting.

\section{Preliminaries}

\subsection{Frames in Hilbert spaces} Let $\Omega$ a countable index set.
A sequence $\{x_j\}_{j
\in \Omega}$ is a \emph{frame} for a Hilbert space $\cH$ if there exist constants $A,B>0$ such that for all $y\in \cH$ the following inequality holds
\begin{equation*}
    A\|y\|^2 \leq \sum_{j\in \Omega} |\langle y,x_j
    \rangle|^2 \leq B \|y\|^2.
\end{equation*}

We call $\{x_j\}_{j
\in \Omega}$ a \emph{Parseval frame} if we can choose $A=B=1$, i.e., if for all $y\in \cH$ we have
\begin{equation*}
    \sum_{j\in \Omega} |\langle y,x_j
    \rangle|^2 = \|y\|^2.
\end{equation*}

The \emph{synthesis operator} or the \emph{reconstruction operator} associated to the frame $\{x_j\}_{j\in \Omega}$ is the map (usually) defined as $R:\ell^2(\Omega) \to \cH$,
\begin{equation*}
    R(\{a_j\}) = \sum_{i\in \Omega} a_j x_j.
\end{equation*}

Throughout the article, 
we say that an operator $T \in \cB(\cH)$ \emph{admits a frame of orbits} if there exists a countable set $\{x_i\}_{i\in I} \subseteq \cH$ such that $\{T^n x_i\}_{i\in I,n\geq0}$ is a frame of $\cH$.

We now summarize  some known related results. Recall that an operator $T\in \cB(\cH)$ is \emph{strongly stable} if $\|T^{n}y\| \to 0$ as $n\to \infty$ for all $y\in \cH$. We refer the reader to  \cite{ACNP25,AP17,CMPP} for the proofs of the following (and related) results. 

\begin{theorem}\label{TadmitFrame}
Let $T \in \cB(\cH)$. The following conditions are equivalent:
\begin{enumerate}
\item[i)]  $T$ admits a frame,
\item[ii)] $T$ is similar to a contraction, and $T^*$ is strongly stable.
\end{enumerate}
\end{theorem}

For Parseval frames we have the following.
\begin{theorem}\label{thm:PF_orbits}
Let $T \in \cB(\cH)$. The following conditions are equivalent:
\begin{enumerate}
\item[i)]  $T$ admits a frame a Parseval frame,
\item[ii)] $T$ is a contraction and $T^*$ is strongly stable.
\end{enumerate}
\end{theorem}

\begin{remark}\label{ss-unitarynoframe}

\noindent
\begin{enumerate} 

\item \label{lem:ss_bon} 
For power bounded operator, the strong stability is determined by the behaviour of the orbits of an orthonormal basis. More precisely, if  $\cH$ is a Hilbert space with orthonormal basis $\{e_j\}$ and $T \in B(\cH)$ is a power bounded  operator, then $T$ is strongly stable if and only if $\norm{T^n e_j} \to 0$ as $n \to +\infty$ for all $j \in \N$.

\item  As a consequence of Theorems \ref{TadmitFrame} and \ref{thm:PF_orbits} we get that unitary operators do not admit frames of orbits. This fact can also be derived from \cite[Proposition~2.1]{CH23}.
\end{enumerate}
\end{remark}

For the sake of completeness, we recall a sufficient condition for an operator to admit a frame with finitely many orbits. This was proved in \cite[Theorem 3.9]{CMPP}.

\begin{theorem}\label{finiteorbits-spec}
    Let $I$ be a finite set and assume that $\{T^n f_i\}_{i\in I, n\geq 0}$ is a frame for $\cH$. Then ${\rm codim(rank}(T^* - \lambda)) = 0$ for all $\lambda \in \D\setminus\Delta$, where $\Delta \subset \D$ is discrete in $\D$. In particular, either $\sigma(T^*) = \overline{\D}$ and $\sigma_p(T^*)=\D$, or $\sigma(T^*)$ is discrete in $\D$.
\end{theorem}

We recall that $\cH$ is a \emph{reproducing kernel Hilbert space} if $\cH$ is a vector subspace of the set of all complex-valued functions defined on a set $X$, and for every $x\in X$ the evaluation functional $e_{x}: \cH \to \C$, $e_x(\phi) = \phi(x)$ is bounded. As a consequence of the Riesz representation theorem, for every $x\in X$ there exists a unique $k_x \in \cH$ such that $$ \phi(x) = e_x(\phi) = \langle \phi,k_x \rangle.$$ The function $k_x$ is called the \emph{reproducing kernel for the point $x$}, and the function $K:X\times X \to \C$ given by $K(x,y) = k_y(x)$ is the \emph{reproducing kernel} for $\cH$.

In this context, if $\{\phi_j\}_{j\in \Omega}$ is a frame for $\cH$, then  the reproducing kernel satisfies
\begin{equation}\label{frame_kernel}
   A \|k_x\|^2 \leq \sum_{j\in \Omega} |\phi_{j}(x)|^2 \leq B \|k_x\|^2
\end{equation}
for every $x\in X$. In particular, if $\{\phi_j\}_{j\in \Omega}$ is a Parseval frame we have for every $x\in X$ that
\begin{equation*}
     \sum_{j\in \Omega} |\phi_{j}(x)|^2 = \|k_x\|^2.
\end{equation*}
Furthermore,  $\{\phi_j\}_{j\in \Omega}$ is a Parseval frame if and only if $K(x,y) = \sum_{j\in \Omega} \phi_j(x) \overline{\phi_{j}(y)}$ for all $x,y \in X$ (see \cite{Paul}).

\subsection{Hardy spaces in the polydisc}\label{Hardy}

Let $N \in \N$. The $N$-dimensional  polydisc and   torus are respectively defined as$$\D^{N} = \{ z \in \C^N \, : \, |z_{j}| < 1 \text{ for all } j \}$$and$$\T^{N} = \{ \zeta \in \C^N \, : \, |\zeta_{j}| = 1 \text{ for all } j \}.$$ We endow the compact abelian group $\T^N$ with its normalized Haar measure $m_N$ (so that $m_N(\T^N) = 1$), and simply write $d\zeta$ for $dm_N(\zeta)$. Note that this measure coincides with the normalized Lebesgue measure of $\T^N$. As usual, we write  $w^{\alpha}$ for the monomial $w_{1}^{\alpha_{1}}\cdots w_N^{\alpha_N}$   ($w\in \C^N$ and $\alpha\in  \Z^N$).

The \emph{Hardy space on the polydisc} $H^{2}(\D^N)$ consists of all holomorphic functions $f$ on $\D^{N}$ such that $$\|f\|_2 = \lim_{r \to 1^{-}} \left( \int_{\T^N} |f(r\zeta)|^2 \, d\zeta \right)^{1/2} < \infty.$$Alternatively, $H^{2}(\D^N)$ can be identified with the space of all power series $$f(z) = \sum_{\alpha \in \N_{0}^N} c_{\alpha}(f) z^{\alpha}$$ having square-summable Taylor coefficients, with the norm $\|f\|_{2} = \big( \sum_{\alpha} |c_{\alpha}(f)|^2 \big)^{1/2}$. It is a well-known fact that $H^2(\D^N)$ is a reproducing kernel Hilbert space with kernel
\begin{equation}\label{kernel}
K(z,w) = k_{w}(z) = \prod_{j=1}^{N} \frac{1}{1 - \overline{w_j} z_j}, \qquad z,w \in \D^N . 
\end{equation}
Note also that \begin{equation}\label{eq-norm-kernel}
\|k_{w}\|^2=\prod_{k=1}^{N} \frac 1 {1 - |w_k|^2}.
\end{equation}

Given  $f\in L^{2}(\T^N)$ (the space of all square-integrable functions) we consider its  Fourier series $$\sum_{\alpha \in \Z^N} \widehat{f}(\alpha) \zeta^{\alpha},$$ where $$\widehat{f}(\alpha) = \int_{\T^N} f(\zeta) \zeta^{-\alpha} \, d\zeta.$$ The \emph{Hardy space on the torus}, $H^2(\T^N)$, is then defined as the closed subspace of $L^2(\T^N)$ consisting of functions whose Fourier coefficients vanish whenever any index $\alpha_j$ is strictly negative. Namely, 
$$
H^2(\T^N) = \Big\{ f \in  L^2 (\T^N)\,\,\colon\,\, \widehat{f}(\alpha) = 0 \,  , \, \,\, \,
\forall \alpha \in \mathbb{Z}^{(\mathbb{N})} \setminus \mathbb{N}_{0}^{(\mathbb{N})} \Big\}.
$$

We recall that, By Fatou's theorem, every $f \in H^2(\D^N)$ has a radial limit $$F(\zeta) = \lim_{r \to 1^{-}} f(r\zeta)$$ for almost every $\zeta \in \T^N$, and $f^* \in H^2(\T^N)$. Conversely, any function $F \in H^2(\T^N)$ can be uniquely extended to a holomorphic function $f \in H^2(\D^N)$ via the Poisson integral. This naturally identifies $H^2(\T^N)$ and $H^2(\D^N)$ isometrically (see, e.g., \cite[Theorem 2.1.3]{Rud}). We remark that, in this identification, Taylor coefficients of $f$ correspond to Fourier coefficients of $F$ (and, since we are in the Hilbert space case, we can establish the isometric isomorphism without the need of Fatou's theorem and Poisson integrals). 
For convenience, we will denote by $f$ both analytic function defined on $\D^N$ and its radial limit defined on $\T^N$ and, when needed,  explicitly mention the domain. 

Finally, we denote by $\bP$ the \emph{Riesz projection}, i.e., the orthogonal projection of $L^{2}(\T^N)$ onto $H^2(\T^N)$. This projection  simply maps $f \in L^2(\T^N)$ to its \emph{analytic} part by keeping  only the Fourier coefficients of $f$ with non-negative exponents:    
$$\bP(f) = \sum_{\alpha \in \N_0^N} \widehat{f}(\alpha) \zeta^{\alpha}.$$

\section{Frames of orbits of multiplication operators}

The goal of this paper is to investigate the existence of frames for the Hardy space $H^{2}(\T^N)$ generated by (finitely or infinitely many) orbits of multiplication operators. 

For $N \in \N$ and a \emph{symbol} $\varphi \in H^{\infty}(\T^N)$, let $M_{\varphi}: H^2(\T^N) \to H^2(\T^N)$ be the multiplication operator defined by $M_{\varphi}f = \varphi f$. It is a standard fact that $M_{\varphi}$ is bounded with norm $\|M_{\varphi}\| = \|\varphi\|_{\infty}$. Also,  $M_{\varphi}$ is an isometry if and only if its symbol $\varphi$ is an inner function, which means that $|\varphi(\zeta)| = 1$ for almost every $\zeta \in \T^{N}$.

The adjoint operator $M_{\varphi}^{*}: H^{2}(\T^N) \to H^{2}(\T^N)$ takes the form $M_{\varphi}^{*}f = \bP(\overline{\varphi}f)$, where $\bP$ is the orthogonal Riesz projection defined above.

 Given $\varphi \in H^{\infty}(\T^N)$, we want to determine whether there exists a countable set $\{f_i\}_{i\in I}\subseteq H^{2}(\T^{N})$ and constants $A,B>0$ such that 
\begin{equation}\label{frameMphi}
    A\|g\|^2 \leq \sum_{n=0}^{\infty}\sum_{i\in I} |\langle g, M_{\varphi}^n f_i \rangle|^2 \leq B\|g\|^2
\end{equation}
holds for all $g \in H^{2}(\T^N)$. A key observation is that the frame property \eqref{frameMphi} can be transferred from $H^{2}(\T^N)$ to $H^{2}(\D^N)$ (and vice versa) via the isometric isomorphism described in Section \ref{Hardy}. Furthermore, recalling that \eqref{kernel} is the reproducing kernel for $H^{2}(\D^N)$, if the set $\{M_{\varphi}^n f_i\}_{i\in I,n\geq 0}$ is a frame for $H^{2}(\D^{N})$, then the frame inequality applied to the reproducing kernel yields the following simple Lemma. 
\begin{lemma}\label{Lem-frameykernel} 
Let $\varphi \in H^{\infty}(\D^N)$ be such that $|\varphi(z)| < 1$ for all $z \in \D^N$. Suppose that $\{M_{\varphi}^n f_i\}_{i\in I,n\geq 0}$ is a frame for $H^{2}(\D^{N})$ with bounds $A, B > 0$. Then, for every $z \in \D^N$, we have
\begin{equation}\label{eq-frameykernel}
A \frac{1-|\varphi(z)|^2}{\prod_{k=1}^{N}(1-|z_k|^2)} \leq \sum_{i\in I} |f_i(z)|^2 \le B \frac{1-|\varphi(z)|^2}{\prod_{k=1}^{N}(1-|z_k|^2)}.  
\end{equation}
\end{lemma}

\begin{proof} 
For  $z \in \D^N$, the frame condition \eqref{frameMphi} applied to  the reproducing kernel $k_z$ gives
$$A \|k_z\|^2 \leq \sum_{n \ge 0} \sum_{i\in I} |\langle k_z, M_{\varphi}^n f_i\rangle|^2 \leq B \|k_z\|^2.$$
By the reproducing property,
the term in the center becomes
$$\sum_{n \ge 0} \sum_{i\in I} |\langle k_z, M_{\varphi}^n f_i \rangle|^2 
= \left( \sum_{n \ge 0} |\varphi(z)|^{2n} \right) \left( \sum_{i\in I}  |f_i(z)|^2 \right) = \frac{1}{1 - |\varphi(z)|^2} \sum_{i\in I}  |f_i(z)|^2,$$
where the geometric series converges since $|\varphi(z)| < 1$. 
We recall from \eqref{eq-norm-kernel} that   $$\|k_z\|^2 = \prod_{k=1}^{N} \frac 1 {1 - |z_k|^2}.$$ 
By substituting these expressions into the frame inequality \eqref{eq-frameykernel} we get  the desired result.~\end{proof}

\begin{remark}\label{casoN>2}
Two of our main results establish that, for $N>1$, it is impossible to construct frames for $H^2(\T^N)$ from finitely many orbits of either $M_{\varphi}$ or $M_{\varphi}^{*}$ (see Proposition \ref{prop-no-finitos} and Theorem \ref{teo-phi1} below). The proof proceeds by showing that there are no such frames in the two-variable case. The case $N > 2$ can be reduced to the two-variable case by considering the orthogonal projection $P_V$ onto the closed subspace $V = \overline{\text{span}}\{\zeta_1^k \zeta_2^l : k,l\geq 0\}$. If $\{M_{\varphi}^{n} f_i\}_{i \in I, n \geq 0}$ (or $\{(M_{\varphi}^{*})^{n} f_i\}_{i \in I, n \geq 0}$) were a frame for $H^2(\T^N)$, its image under $P_V$ would yield a frame for $V$. Since we show that this does not occur for $N=2$, the non-existence of such frames for $H^2(\T^N)$ follows.
\end{remark}

\subsection{Main results}

In this section, we state and prove our main results. We state them  for $H^2(\mathbb{T}^N)$, but they are clearly valid also for $H^2(\mathbb{D}^N)$. Moreover, throughout the proofs we move freely between these spaces in order to exploit tools from  either Fourier or complex analysis.

\begin{theorem}\label{frame_varphi<=1}
Let $N\in \N$ and $\varphi \in H^{\infty}(\T^N)$. The following are equivalent. 
\begin{enumerate}
\item[i)] The operator $M_{\varphi}$ admits a Parseval frame of $H^2(\T^N)$;
\item[ii)] The operator $M_{\varphi}$ admits a  frame of $H^2(\T^N)$;
\item[iii)] The function $\varphi$ satisfies $\|\varphi\|_{\infty} \leq 1$ and $|\widehat{\varphi}(0)| < 1$.
\end{enumerate}
\end{theorem}

\begin{proof}
The implication $i) \Rightarrow ii)$ is clear, so let us show $ii) \Rightarrow iii)$. 
    Suppose that $\{ M_{\varphi}^n f_i \}_{i\in I,n\geq0}$ is a frame for $H^{2}(\T^N)$ for some countable set $\{f_i\}_{i\in I} \subseteq H^{2}(\T^N)$. 
    {Following the approach in \cite[Remark 3.8]{ACCP23}, the synthesis operator of the frame $\{ M_{\varphi}^n f_i \}_{i\in I,n\geq0}$ can be alternatively defined on $H^2_{\ell^2(I)}(\T)$ (the Hardy space of $\ell^2(I)$-valued functions)—rather than the usual sequence space $\ell^2(\mathbb{N}_0\times I)$— by means of a precomposition with the isometric isomorphism $g \mapsto \{\langle g, S^n e_i \rangle \}_{i\in I, n\geq0}$.} With this in mind, the synthesis operator $R: H^{2}_{\ell^2(I)}(\T) \to H^2(\T^N)$ takes the form
\begin{equation*}
    Rg = \sum_{n=0}^{\infty} \sum_{i\in I} \langle g, S^n e_i \rangle M_{\varphi}^n f_{i},
\end{equation*}
where $S$ is the unilateral shift on $H^2_{\ell^2(I)}(\T)$ and $\{e_i\}_{i \in I}$ is the standard orthonormal basis of $\ell^2(I)$. 
The operator $R$ is bounded, surjective, and intertwines $M_\varphi$ and $S$, i.e.,  $M_{\varphi}^m R = R S^m$ for all $m\in \N$. Applying this to the basis elements $S^k e_j$, we obtain
\begin{equation*}
    \varphi^m R(S^k e_j) = R (S^{m+k} e_j).
\end{equation*}
Assume, towards a contradiction, that $\|\varphi\|_{\infty} > 1$. Then, there exists $\lambda > 1$ such that the set $E_\lambda=\{\zeta\in \T^N : |\varphi(\zeta)|> \lambda \}$ has positive Lebesgue measure. Then, we have
\begin{align*}
    \lambda^{2m} \int_{E_ \lambda} |R(S^k e_j)(\zeta)|^2 \, d\zeta
    &\leq \int_{E_\lambda} |\varphi(\zeta)|^{2m} |R(S^k e_j) (\zeta)|^2 \, d\zeta \\
    & \leq \int_{\T^N} |R (S^{m+k} e_j) (\zeta)|^2 \, d\zeta \leq \|R\|^2.
\end{align*}
Since $\lambda > 1$, letting $m \to +\infty$ we get that necessarily $\int_{E_ \lambda} |R(S^k e_j)(\zeta)|^2 \, d\zeta = 0$. Hence, $R(S^k e_j) = 0$ almost everywhere on $E_\lambda$ for all $k \in \N, j \in I$. Since $R$ is surjective and $\{S^k e_j\}_{j\in I,k\geq0}$ spans $H^2_{\ell^2(I)}(\T)$, this implies that every function in $H^2(\T^N)$ vanishes almost everywhere on $E_\lambda$, which is impossible since $|E_\lambda| > 0$. This shows that $\|\varphi\|_\infty \le 1$.
    
To finish proving $iii)$, note that if $\|\varphi\|_\infty \le 1$ and $|\widehat\varphi(0)| = 1$, then $\varphi$ must be a unimodular constant. In this case, $M_\varphi$ is a unitary operator, which cannot admit a frame of orbits (as observed in (2) of Remark \ref{ss-unitarynoframe}). Hence, $|\widehat{\varphi}(0)| < 1$.

It remains to show $iii) \Rightarrow i)$. Assume $\|\varphi\|_{\infty} \leq 1$ and $|\widehat{\varphi}(0)| < 1$. Since $M_\varphi$ is a contraction, by Theorems \ref{TadmitFrame} and \ref{thm:PF_orbits}, it suffices to show that $\norm{(M_\varphi^*)^n g} \to 0$ for all $g \in H^2(\T^N)$.
    
Let $\mathbf{i} \in \N_0^N$ and consider the monomial $$\rchi^\mathbf{i}(z) = z^\mathbf{i}=z_1^{\mathbf{i}_1} \cdots z_N^{\mathbf{i}_N}.$$ By (1) of Remark \ref{ss-unitarynoframe}, it is enough to show that $\norm{(M_\varphi^*)^n \rchi^\mathbf{i}} \to 0$ as $n \to +\infty$. An easy computation shows  that  $(M_\varphi^*)^n f = \bP (\overline{\varphi}^nf)$, where $\bP$ is the Riesz projection. Therefore, we need to analyze $\bP(\overline{\varphi}^n \rchi^\mathbf{i})$.
    
 Write $\varphi(z) = a_0 + q(z) + \eta(z)$, where $a_0 = \widehat{\varphi}(0)$, $q(z)$ contain exactly $z^\mathbf{j}$ terms with $0 < \mathbf{j} \le \mathbf{i}$, and $\eta(z)$ contains the remaining $z^\mathbf{j}$ terms with $\mathbf{j} \not\le \mathbf{i}$.
Note that
\begin{align*}
(\overline{\varphi})^n = (\overline{a_0} + \overline{q} + \overline{\eta})^n = (\overline{a_0} +\overline{q})^n + \overline{\eta} \sum_{r=1}^n \binom{n}{r} (\overline{a_0} + \overline{q})^{n-r} \overline{\eta}^{r-1} = (\overline{a_0} + \overline{q})^n + \overline{\eta} \overline{\lambda}_n 
\end{align*}
with $\lambda_n \in H^\infty(\T^N)$. Since every monomial in $\overline{\eta} \overline{\lambda}_n $ contains a conjugate variable to a power strictly greater than the corresponding one in $\rchi^\mathbf{i}$, we have $\bP(\overline{\varphi}^n \rchi^\mathbf{i}) = \bP((\overline{a_0} + \overline{q})^n \rchi^\mathbf{i})$. 
Note also that $\bP(\overline{q}^s \rchi^\mathbf{i})=0 $ for any $s> |\mathbf{i}|, $ since every monomial in $\overline{q}^s \rchi^\mathbf{i}$ has some surviving conjugate coordinate. This and  the binomial formula give us
\begin{equation*}
    \bP(\overline{\varphi}^n \rchi^\mathbf{i}) = \sum_{s=0}^{|\mathbf{i}|} \binom{n}{s} \overline{a_0}^{n-s} \bP(\overline{q}^s \rchi^\mathbf{i}),
\end{equation*}
and then,
\begin{equation*}
 \norm{\bP(\overline{\varphi}^n \rchi^\mathbf{i})} \le \sum_{s=0}^{|\mathbf{i}|} \binom{n}{s} |a_0|^{n-s} \norm{\bP(\overline{q}^s \rchi^\mathbf{i})}.
\end{equation*}
Since $|a_0| < 1$, each term $\binom{n}{s} |a_0|^{n-s}$ goes to $0$ as $n \to +\infty$. The sum is finite and independent of $n$, so the limit is zero, concluding the proof.
\end{proof}

\begin{remark}
    As an immediate consequence of the previous theorem, constant functions admit (Parseval) frames if and only if they have absolute value less than $1$—a fact that is also easily verified directly.
\end{remark}

We now focus on  the existence of frames for $H^{2}(\T^N)$ generated  by finitely many orbits of $M_\varphi$. First, we show that for $N>1$ there are no such frames.

\begin{proposition}\label{prop-no-finitos}
Let $N>1$ and $\varphi \in H^{\infty}(\T^N)$. Then, $M_{\varphi}$ does not admit a frame with finitely many orbits.
\end{proposition}

\begin{proof}
By Remark \ref{casoN>2}, it is enough to prove the statement for $N=2$. Suppose  that
there exists $\{f_i\}_{i=1}^m \subseteq H^{2}(\T^2)$ such that $\{M_{\varphi}^{n} f_i\}_{1 \leq i \leq m,n\geq0}$ is a frame for $H^{2}(\T^2)$. Then, the extensions of $\{M_{\varphi}^{n} f_i\}_{1 \leq i \leq m,n\geq0}$ as holomorphic functions on the polydisc form a frame for $H^{2}(\D^2)$. 

We split the proof into two cases: either $|\varphi(\zeta)|<1$ in a set  $E\subseteq \T^2$ of positive measure or  $|\varphi(\zeta)|=1$ for almost every  $\zeta \in \T^{2}$ .

We first suppose $E=\{|\varphi|<1\}$ has positive measure. 
We use  Lemma \ref{Lem-frameykernel} to get, for $0<r<1$ and almost every $(\zeta_1,\zeta_2) \in E$ (those where all radial limits exist), 
$$ \sum_{i=1}^m |f_i(r\zeta_1,r\zeta_2)|^2 \geq A \frac{1-|\varphi(r\zeta_1,r\zeta_2)|^2}{ (1-r^2)^{2}} \to +\infty$$
when $r\to 1^{-}$. Since  $E$ has positive measure, this means that for some $i=1,...m$, the function $f_i $ does not belong to  $ H^{2}(\T^2)$ , a contradiction.

Let us suppose now that $|\varphi(\zeta)|=1$ for almost every  $\zeta \in \T^{2}$. First, we show that $\varphi$ cannot depend solely on a single variable.

Assume, for the sake of contradiction, that $\varphi(z_1,z_2) = \widetilde{\varphi}(z_1)$ for all $z_2 \in \D$, and define the closed subspace $W = \overline{\text{span}\{z_2^k : k\geq 0\}} \subseteq H^2(\D^2)$. Since $\{M_{\varphi}^n f_i\}_{n\geq 0,1 \leq i \leq m}$ is a frame for $H^2(\D^2)$, its orthogonal projection onto $W$, namely $\{P_{W}(M_{\varphi}^n f_i)\}_{1 \leq i \leq m,n\geq 0}$, must be a frame for $W$. Consequently, this projected system must be complete in $W$. However, observe that\begin{equation*}P_{W}(M_{\varphi}^n f_i)(z_1,z_2)= M_{\varphi}^n f_i(0,z_2) = \varphi(0,z_2)^n f_i(0,z_2)= \widetilde{\varphi}(0)^n f_i(0,z_2),\end{equation*}which implies that $\dim\left(\overline{\text{span}\{P_{W}(M_{\varphi}^n f_i)\}_{n\geq 0, 1 \leq i \leq m}}\right) \leq m < \infty = \dim(W)$. This is a contradiction, hence $\varphi$ must depend on both variables.

By \cite[Proposition 2.8]{Mar15}, the limit $$\lim_{r\to 1^-} \varphi(r\zeta_1,0)$$ exists for almost every  $\zeta_1 \in \T$, and the function $h(\zeta_1) = \lim_{r\to 1^-} \varphi(r\zeta_1,0)$ belongs to $H^2(\T)$. We can see, by computing the Fourier coefficients of  both sides, that \begin{equation*}h(\zeta_1) = \int_{\T} \varphi(\zeta_1,\zeta_2)\, d\zeta_2.\end{equation*} 
Since $|\varphi(\zeta_1,\zeta_2)| = 1$ for almost every  $\zeta_1, \zeta_2 \in \T$
we can see that $|h|<1$ in a set of postive measure (actually, it can be seen that $|h|<1$ almost everywhere, but we do not need this fact). Indeed, if $|h(\zeta_1)| = 1$ for almost every  $\zeta_1 \in \T$, then 
\begin{align*}
\int_{\T} |\varphi(\zeta_1,\zeta_2) - h(\zeta_1)|^2 & \,d\zeta_2
= \int_{\T} \left( |\varphi(\zeta_1,\zeta_2)|^2 - h(\zeta_1) \overline{\varphi(\zeta_1,\zeta_2)} - \overline{h(\zeta_1)} \varphi(\zeta_1,\zeta_2) + |h(\zeta_1)|^2 \right) \,d\zeta_2 \\
&= \int_{\T} |\varphi(\zeta_1,\zeta_2)|^2 \, d\zeta_2 - h(\zeta_1) \overline{\int_{\T}\varphi(\zeta_1,\zeta_2) \,d\zeta_2} - \overline{h(\zeta_1)} \int_{\T} \varphi(\zeta_1,\zeta_2) \,d\zeta_2 + 1 \\
&= 1 - |h(\zeta_1)|^2 - |h(\zeta_1)|^2 + 1 \\
&= 2 - 2|h(\zeta_1)|^2 = 0.
\end{align*}
This means that $\varphi(\zeta_1,\zeta_2) = h(\zeta_1)$ for almost every  $\zeta_2 \in \T$, which contradicts our previous remark on  $\varphi$.
Thus,  the set  $$F=\{\zeta\in \T\colon |h(\zeta)|=\lim_{r\to1^{-}}|\varphi(r\zeta_1,0)|<1\}$$ has positive measure.

For $0<r<1$ and $\zeta_1 \in \T$, by Lemma \ref{Lem-frameykernel} we have
\begin{equation*}
\sum_{i=1}^{m} |f_i(r \zeta_1,0)|^2 \geq A \frac{1 -|\varphi(r \zeta_1,0)|^2 } {(1- |r|^2) }, 
\end{equation*}
where $A>0$ is the lower frame bound.
Defining $h_i(\zeta_1) = \lim_{r\to 1^-} f_i(r\zeta_1,0)$, we have   $$\sum_{i=1}^{m} |h_i(\zeta_1)|^2 = \lim_{r \to 1^{-}} \sum_{i=1}^{m} |f_i(r \zeta_1,0)|^2  \geq A \lim_{r \to 1^{-}} \frac{1 -|\varphi(r \zeta_1,0)|^2 } {1- |r|^2} = +\infty$$ for almost every $\zeta_1 \in F$. This is a contradiction, since by  \cite[Proposition 2.8]{Mar15} each $h_i$ belongs to $H^2(\T)$.
\end{proof}

For the  case $N=1$ the situation is different. This is shown in the following theorem,  part of whose proof  relies on  classical facts on model spaces and Wold decomposition.

\begin{theorem}\label{teo-frames-N=1}
Let $\varphi\in H^{\infty}(\T)$. Then the following statements are equivalent.
\begin{enumerate}
\item[i)] The function $\varphi$ is a finite Blaschke product;
\item[ii)] The operator $M_{\varphi}$ admits an orthonormal basis of finitely many orbits;
\item[iii)] The operator $M_{\varphi}$ admits a Parseval frame of finitely many orbits;
\item[iv)] The operator $M_{\varphi}$ admits a frame of finitely many orbits. 
\end{enumerate}
\end{theorem}

\begin{proof}
We first observe that $ii) \Rightarrow iii)$ and $ \Rightarrow iii) \Rightarrow iv)$ are straightforward,
so we have to show $i) \Rightarrow ii)$ and $iv) \Rightarrow i)$.

$i) \Rightarrow ii)$ Suppose that $\varphi$ is a finite Blaschke product and define $W = H^{2}(\T) \ominus \varphi H^{2}(\T)$. By  \cite[Proposition 5.19]{GMR} we know that $\dim(W) = m <\infty$. Since $\varphi$ is, in particular, an inner function, $M_\varphi$ is a pure 
isometry, i.e., $\bigcap_{n \geq 0} \varphi^n H^2(\mathbb{T}) = \{0\}$. This means  
 that the so-called Wold decomposition of $H^2(\T)$ takes the form
$$H^2(\T) = \bigoplus_{n=0}^{\infty} \varphi^n W = \bigoplus_{n=0}^{\infty} M_{\varphi}^n W,$$
where the sum is orthogonal (see \cite{Ste}). 
Thus, if we choose an orthonormal basis $\{f_1, ...,f_m\}$ of $W$, we obtain that $\{M_{\varphi}^n f_i\}_{1 \leq i \leq m,n\geq 0}$ is an orthonormal basis of $H^{2}(\T)$.

$iv) \Rightarrow i)$ Assume that $M_\varphi$ admits a frame of finitely many orbits, this is, there exist $\{f_i\}_{i=1}^{m} \subset H^2({\T})$ such that $\{M_{\varphi}^n f_i\}_{1\leq i \leq m,n\geq 0} $ is a frame for $ H^2({\T})$.

We see first that $iv)$ implies that $\varphi$ is an inner function, i.e., $|\varphi(\zeta)|=1$  for almost every $\zeta\in\T$. Let  $E\subset \T$ be the set where $|\varphi|<1$. 

For each $r>0$ and $\zeta \in E$, Lemma \ref{Lem-frameykernel} gives us
$$ \sum_{i=1}^m |f_i(r\zeta)|^2 \geq A \frac{1-|\varphi(r\zeta)|^2}{(1-r^2)}.$$
Since $|\varphi\zeta)|<1$, the last expression  goes to infinity as $r\to 1^{-}$ for almost every $\zeta\in E$. This means that the set $E$ must  have measure zero, since $\sum_{1\le i\le m}|f_i|^2$ has finite integral on $\T$.  Hence, $\varphi$ is an inner function. Finally, since $W = H^{2}(\T) \ominus \varphi H^2(T)$ is a finite dimensional model space,  by \cite[Proposition 5.19]{GMR} we know that $\varphi$ must be a finite Blaschke product.
\end{proof}

\begin{example}
Let us see what the orbits of $M_\varphi$ look like when $\varphi$ is a finite Blaschke product. First, we take $a\in \D$ and consider the Blaschke factor 
$$\varphi_a (\zeta) = \frac{\zeta-a}{1-\overline{a}\zeta}, \quad \zeta\in \T.$$ 
In this case, the functions $f$ such that $\{M_{\varphi_a}^n f\}_{n\ge 0}$ is an orthonormal basis are precisely those of the form 
$$f(\zeta) = c \frac{\sqrt{1-|a|^2}}{1 - \overline{a} \zeta}$$
for some $c\in \T$.
This can be verified by a direct computation via reproducing kernels. Indeed, we know that the collection $\{M_{\varphi_a}^n f\}_{n\ge 0}$ is a Parseval frame if and only if
$$\frac{1}{1-\overline{w}z} = \sum_{n=0}^{\infty} \varphi_a^n(z) f(z) \overline{\varphi_a^n(w)} \overline{f(w)} = \frac{f(z)\overline{f(w)}}{1 - \varphi_a(z)\overline{\varphi_a(w)}},$$
which algebraically simplifies to the expression of $f$. 

Alternatively, it can also be deduced from the theory of model spaces. Since $\varphi_a$ is an inner function, we can consider the associated model space $N_{\varphi_a} = H^2(\T) \ominus \varphi_a H^2(\T)$. Now, $H^2(\T)$ admits the orthogonal direct sum decomposition
$$H^2(\T) = \bigoplus_{n=0}^{\infty} \varphi_a^n N_{\varphi_a}.$$
Furthermore, for a single Blaschke factor, the model space $N_{\varphi_a}$ is exactly one-dimensional and is spanned by the normalized reproducing kernel at the point $a$. That is, $N_{\varphi_a} = \text{span}\{f\}$, where $f$ is given as above.

Since the multiplication operator $M_{\varphi_a}$ is an isometry (because $|\varphi_a(\zeta)| = 1$ almost everywhere on $\T$) and $\|f\| = 1$, the orthogonal decomposition immediately implies that the sequence $\{M_{\varphi_a}^n f\}_{n\geq 0}$ forms an orthonormal basis for $H^2(\T)$.

This construction naturally extends to any finite Blaschke product. Let 
$$B(\zeta) = \prod_{j=1}^d \varphi_{a_j}(\zeta)$$ 
be a Blaschke product of degree $d \geq 1$, where $a_1, \dots, a_d \in \D$ are repeated according to multiplicity. Since $B$ is an inner function, we can consider the corresponding model space $N_B = H^2(\T) \ominus B H^2(\T)$. By the Wold decomposition, $H^2(\T)$ can be orthogonally decomposed as
$$H^2(\T) = \bigoplus_{n=0}^{\infty} B^n N_B.$$

It is a well-known fact that the dimension of $N_B$ equals the degree of the finite Blaschke product (or equivalently, its number of zeros counted with multiplicity), so $\dim(N_B) = d$ \cite[Proposition 5.16]{GMR}. Therefore,  the orbits by (the isometry) $M_B$ of any orthonormal basis $\{e_1, \dots, e_d\}$ for the finite-dimensional space $N_B$  generate an orthonormal basis for $H^2(\T)$. 
A classical and explicit choice for such a basis is the Malmquist-Walsh basis (see, for example, \cite[Section 5.9]{GMR}), which is defined by
$$e_1(\zeta) = \frac{\sqrt{1-|a_1|^2}}{1-\overline{a_1}\zeta}, \quad \text{and} \quad e_k(\zeta) = \frac{\sqrt{1-|a_k|^2}}{1-\overline{a_k} \zeta} \prod_{j=1}^{k-1} \varphi_{a_j}(\zeta) \quad \text{for } 2 \leq k \leq d.$$
As a consequence of all this, the collection 
$$\{M_B^n e_k : n \geq 0, \, 1 \leq k \leq d\}$$ 
forms an orthonormal basis for $H^2(\T)$. This illustrates that for a finite Blaschke product of degree $d$, the orthonormal basis is generated by exactly $d$ functions.
\end{example}

As we can see from the proof of Theorem \ref{teo-frames-N=1}, if $|\varphi| $ is not equal to 1 almost everywhere in $\T$, then the orbit of any $f\in\Hardy$ does not satisfy the lower frame  bound. The following example shows that, loosely speaking,  this may be the only  obstruction preventing the orbit from being a frame. Moreover, this shows that the orbit of a cyclic vector need not be a frame, a fact that appears to contradict Proposition 3.2 in \cite{Ches}.

\begin{example}
    Let  $\varphi\in H^{\infty}(\T)$ be given as $\varphi(\zeta) = ({1+\zeta})/{2}$. Then,  $\{M_{\varphi}^n 1\}_{n\geq 0} = \{\varphi^n\}_{n\geq 0}$ is a complete Bessel sequence which cannot be a frame by Theorem \ref{teo-frames-N=1}. 
\end{example}

Indeed, let us write
\begin{align*}
\varphi^n(\zeta) = \left( \frac{1+\zeta}{2} \right)^n = \frac{1}{2^n} \sum_{k=0}^{n} \binom{n}{k} \zeta^k.
\end{align*}
By induction,  if $g\in H^{2}(\T)$ satisfies $\langle g,\varphi^n \rangle =0$ for all $n\in \N$, we can see that every Fourier coefficient of $g$ is zero and, therefore, $g$ must  be the zero function. This shows that $\{\varphi^n\}_{n\geq 0}$ is complete. To see that it is a Bessel sequence, take $g\in H^{2}(\T)$ with $g(\zeta) = \sum_{j=0}^{\infty} b_j \zeta^j $. Then,
\begin{align*}
 \sum_{n=0}^{\infty} |\langle g,\varphi^n \rangle|^2
&= \sum_{n=0}^{\infty} \left| \frac{1}{2^n} \sum_{j=0}^n \binom{n}{j} b_j \right|^2 = \sum_{n=0}^{\infty} \left| \frac{1}{2^n} \sum_{j=0}^n \binom{n}{j}^{1/2} \binom{n}{j}^{1/2} b_j \right|^2 \\
&\leq \sum_{n=0}^{\infty} \frac{1}{2^{2n}} \left(  \sum_{j=0}^n \binom{n}{j} \right) \left( \sum_{j=0}^n \binom{n}{j} |b_j|^2 \right) \\ 
& \leq \sum_{n=0}^{\infty} \frac{1}{2^{n}} \left( \sum_{j=0}^n \binom{n}{j} |b_j|^2 \right)  \leq \sum_{j=0}^{\infty}  \left( \sum_{n=j}^\infty \frac{1}{2^{n}} \binom{n}{j}  \right) |b_j|^2, \\
&= 2 \|g\|^2
\end{align*}
where the last equality follows from the identity $\sum_{n=j}^{\infty} \binom{n}{j} x^n = {x^j}{(1-x)^{-(j+1)}} $ for $x=1/2$.

Although the previous theorem tells us  that  $\{\varphi^n\}_{n\geq 0}$ cannot be a frame, we sketch an argument showing that it does not satisfy the lower frame bound. For each $m\in \N$, we define $g_m (\zeta) = \sum_{j=0}^{m} (-1)^j \zeta^j $, so that $\|g_m\|= \sqrt{m}$. Using the binomial formula and Pascal's identity, we have
\begin{align*}
\sum_{n=0}^{\infty} |\langle g_m,\varphi^n \rangle|^2
&= \sum_{n=0}^{m}  \left| \frac{1}{2^n} \sum_{j=0}^n \binom{n}{j} (-1)^j \right|^2
+ \sum_{n=m+1}^{\infty}  \left| \frac{1}{2^n} \sum_{j=0}^m \binom{n}{j} (-1)^j \right|^2 \\
&\le  1 + \sum_{n=m+1}^{\infty}  \left|   \frac{1}{2^n} \binom{n-1}{m} \right|^2.
\end{align*}
The last expression can be seen as the sum of the squares of the probabilities of a negative binomial distribution $NB(m+1,1/2)$ and, therefore, is bounded by 1. As a consequence,  the lower frame bound cannot hold. 

\subsection{Frames of orbits of the adjoint of multiplication operators}
We now address the question whether $M_{\varphi}^{*}$  admits frames of (finitely or infinitely many) orbits of  $H^{2}(\D^{N})$. In this case, the answer is simpler, in the sense that it does not depend on $N$.

\begin{theorem}\label{teo-phi1}
	Let $N\in \N$ and $\varphi \in H^\infty (\T^N)$. The following statements are equivalent:
	\begin{enumerate}
		\item[i)] There exists a family of functions $\{g_{i}\}_{i\in I}$ in $H^2(\T^N)$ such that $\{{(M_{\varphi}^*)}^{n}g_i\}_{i\in I, n\geq0}$ is a frame for $H^2(\T^N)$;
		\item[ii)] $|\varphi(\zeta)|<1$  for almost every \ $\zeta\in\T^N$.
	\end{enumerate}
	Furthermore, if these equivalent conditions hold, the index set $I$ must be infinite.
\end{theorem}

\begin{proof}
	By Theorem \ref{TadmitFrame}, statement i) is equivalent to saying that $M_{\varphi}^{*}$ is similar to a contraction and that $M_{\varphi}$ is strongly stable. The latter condition means that
	$$\|M_{\varphi}^n g \|^2=\int_{\T^N} |\varphi(\zeta)|^{2n} |g(\zeta)|^2 \,d\zeta \to 0$$
	as $n \to \infty$, for every $g\in H^2(\T^{N})$. In particular, taking $g \equiv 1$, the strong stability of $M_{\varphi}$ implies that $\int_{\T^N}|\varphi(\zeta)|^{2n} \,d\zeta \to 0$ as $n\to \infty$. Now, if $|\varphi(\zeta)|=1$ for every $\zeta$ in a set $E\subseteq \T^{N}$ of positive measure, then 
	$$\int_{\T^N}|\varphi(\zeta)|^{2n} \,d\zeta \geq \int_{E} |\varphi(\zeta)|^{2n} \,d\zeta = |E| > 0$$ 
	for all $n$, which contradicts the strong stability. Therefore, we must have $|\varphi(\zeta)|<1$ a.e.\ $\zeta\in \T^N$. This proves i) $\Rightarrow$ ii).
	
	The converse is a direct consequence of the Lebesgue Dominated Convergence Theorem.

	Finally, if $I$ were finite, Theorem \ref{finiteorbits-spec} would imply that either $\sigma(M_{\varphi}) = \overline{\D}$ with $\sigma_p(M_{\varphi})=\D$, or $\sigma(M_{\varphi})$ is discrete in $\D$. However, identifying again $H^2(\T^N)$ with $H^2(\D^N)$ via the analytic extension of $\varphi$, it is a well-known result for the analytic Toeplitz operator $M_\varphi$ that $\sigma(M_{\varphi}) = \overline{\varphi(\D^N)}$ 
    {(see \cite[Theorem 3.3.8]{MAR} for the one-variable case, the proof works for any number of variables)} and that $\sigma_{p}(M_{\varphi})$ is either empty or a singleton (if $\varphi$ is constant). This contradiction proves the result.
\end{proof}

\section{Hardy spaces of Dirichlet series}

In this section, we see that our previous results for $H^2(\mathbb{T}^N)$ ($N>1$) extend to  the infinite-dimensional torus $\mathbb{T}^\infty$ and, consequently, to the Hardy spaces of Dirichlet series $\mathcal H_2$. 
As for the $N$-dimensional case, we can endow the compact abelian group $\T^\infty$ with its normalized Haar measure, which can also be seen as the countable product  of the  normalized Lebesgue measure of $\T$. We keep the notation  $\zeta^{\alpha}$ for the monomial $\zeta_{1}^{\alpha_{1}}\cdots \zeta_n^{\alpha_n}\cdots$   (for   $\zeta\in \T^\infty$ and $\alpha\in  \Z^{(\N)}$ an eventually-zero multi-index).
With this notation,  \emph{Hardy space on the infinite torus} $H^2(\T^\infty)$ is defined as 
$$
H^2(\T^\infty) = \Big\{ f \in  L^2 (\T^\infty)\,\,\colon\,\, \widehat{f}(\alpha) = 0 \,  , \, \,\, \,
\forall \alpha \in \mathbb{Z}^{{(\N)}} \setminus \mathbb{N}_{0}^{{(\N)}} \Big\}.
$$
It is easy to check that our results from the previous section hold for the infinite dimensional case. For example, to see that there are no frames with finitely many generators both for the multiplication operator and for its adjoint in dimension greater than one (see  Proposition \ref{prop-no-finitos} and Theorem \ref{teo-phi1}) it was enough to consider the $N=2$ and to use Remark \ref{casoN>2}. This remark is clearly valid if we consider $N=\infty$.

Once our results are established for $H^2(\mathbb{T}^\infty)$, they   transfer to the corresponding  Hardy space of  Dirichlet series $\mathcal H_2$ via the so-called \emph{Bohr's transform}. 

First, let us recall that an \emph{ordinary Dirichlet series} is a function of a complex variable $s$ given by the formal sum
$$D(s) = \sum_{n=1}^\infty a_n n^{-s},$$
where $(a_n)_{n \geq 1}$ is a sequence of complex coefficients. The \emph{Hardy space of Dirichlet series $\mathcal{H}_2$} is defined as the Hilbert space of all such series with square-summable coefficients: $$\sum_{n=1}^\infty |a_n|^2 < \infty.$$ In the remarkable article \cite{hedenmalm1997}, the authors introduced the space  $\mathcal{H}_2$ to solve Beurling's problem on Riesz bases of dilations. They also showed that the space of multipliers of  $\mathcal H_2$ identifies with  $\mathcal{H}_\infty$, the space of Dirichlet series that define a bounded holomorphic function on $\C_+=\{s\in\C \colon \Re(s)>0\}$. Note that the product of formal Dirichlet series is also a Dirichlet series. We refer the reader to the just cited \cite{hedenmalm1997} and to the monographs \cite{DefGarMaeSev19,QueQue20} for details.

As we mentioned above, the connection between $H^2(\mathbb{T}^\infty)$ and the space $\mathcal H_2$ is given by Bohr's transform. Every integer $n \geq 1$ can be uniquely factored into prime numbers as $n = p_1^{\alpha_1} p_2^{\alpha_2} \cdots p_k^{\alpha_k}$, which we write $\mathbf p^\alpha$ for short. By identifying the $j$-th prime number $p_j$ with a complex coordinate variable $\zeta_j$ on the infinite torus, the Bohr's transform maps the term $n^{-s}$ to the multi-variable monomial $\mathbf{\zeta}^\alpha = \zeta_1^{\alpha_1} \zeta_2^{\alpha_2} \cdots \zeta_k^{\alpha_k}$, where $n=\mathbf p^\alpha$.

Consequently, the Dirichlet series $D(s)=\sum_{n=1}^\infty a_n n^{-s}$ is transformed into a Fourier series in infinitely many variables:
$${f}(\zeta) = \sum_{\alpha\in \mathbb N_0^{(\mathbb N)}} a_{\mathbf{p}^\alpha} \zeta^\alpha.$$
Note that,  in particular,  we have $a_1=\hat{f}(0)$. This transformation establishes an isometric isomorphism between the Hardy space of Dirichlet series $\mathcal{H}_2$ and the classical Hardy space on the infinite torus, $H^2(\mathbb{T}^\infty)$. A deeper result is that it is also an isometric isomorphism between $H^\infty(\T^\infty)$ and the space  $\mathcal H_\infty$. Moreover, Bohr's transform is multiplicative. Thanks to this identification and its properties, a Dirichlet series $D_0\in \mathcal H_\infty$ acting as a multiplication operator $M_{D_0}$ on $\mathcal H_2$ corresponds precisely to its image through Bohr's transform in $H^\infty(\mathbb T^\infty)$ acting on $H^2(\mathbb{T}^\infty)$. This allows us to fully translate our study of frames of orbits of multiplication operators to $\mathcal H_2$. Indeed, Theorem \ref{frame_varphi<=1} and Proposition \ref{casoN>2} readily provide the following. 

\begin{theorem}\label{frame_D_0}
Let $D_0 \in \mathcal H_\infty$. The following are equivalent. 
\begin{enumerate}
\item[i)] The operator $M_{D_0}$ admits a Parseval frame of $\mathcal H_2$;
\item[ii)] The operator $M_{D_0}$ admits a  frame of $\mathcal H_2$;
\item[iii)] The series satisfies $\|D_0\|_{\infty} \leq 1$ and $|a_1| < 1$, with $a_1$ the constant term of $D_0$.
\end{enumerate}
When these equivalent conditions hold, the frame necessarily has infinitely many orbits. 
\end{theorem}


\begin{thebibliography}{99}

\bibitem{ACCP23}
\textsc{A. Aguilera, C. Cabrelli, D. Carbajal, and V. Paternostro},
\newblock {\em Frames by orbits of two operators that commute},
\newblock Appl. Comput. Harmon. Anal. 66, 46--61 (2023).

\bibitem{ACNP25}
\textsc{A. Aguilera, C. Cabrelli, F. Negreira, and V. Paternostro},
\newblock {\em Optimal dynamical frames},
\newblock to appear in J. Anal. Math. (2026).

\bibitem{ACCMP17}
\textsc{A. Aldroubi, C. Cabrelli, A. F. \c{C}akmak, U. Molter, and A. Petrosyan},
\newblock {\em Iterative actions of normal operators},
\newblock J. Funct. Anal. 272, 1121--1146 (2017).

\bibitem{ADK13}
\textsc{A. Aldroubi, J. Davis, and I. Krishtal},
\newblock {\em Dynamical sampling: time-space trade-off},
\newblock Appl. Comput. Harmon. Anal. 34, 495--503 (2013).

\bibitem{ADK15}
\textsc{A. Aldroubi, J. Davis, and I. Krishtal},
\newblock {\em Exact reconstruction of signals in evolutionary systems via spatiotemporal trade-off},
\newblock J. Fourier Anal. Appl. 21, 11--31 (2015).

\bibitem{ACMT}
\textsc{A. Aldroubi, C. Cabrelli, U. Molter, and S. Tang},
\newblock {\em Dynamical sampling},
\newblock Appl. Comput. Harmon. Anal. 42, 378--401 (2017).

\bibitem{AP17}
\textsc{A. Aldroubi and A. Petrosyan},
\newblock {\em Dynamical sampling and systems from iterative actions of operators},
\newblock in: I. Pesenson, Q. Le Gia, A. Mayeli, H. Mhaskar, D. X. Zhou (eds.), Frames and Other Bases in Abstract and Function Spaces. Applied and Numerical Harmonic Analysis. Birkh\"auser, Cham, 2017.

\bibitem{CMPP}
\textsc{C. Cabrelli, U. Molter, V. Paternostro, and F. Philipp},
\newblock {\em Dynamical sampling on finite index sets},
\newblock J. Anal. Math. 140, 637--667 (2020).

\bibitem{CMS21}
\textsc{C. Cabrelli, U. Molter, and D. Su\'arez},
\newblock {\em Multi-orbital frames through model spaces},
\newblock Comp. Anal. Oper. Theory 15, 1--22 (2021).

\bibitem{CMS}
\textsc{C. Cabrelli, U. Molter, and D. Su\'arez},
\newblock {\em Frames of iterations and vector-valued model spaces},
\newblock in: S. D. Casey, M. M. Dodson, P. J. S. G. Ferreira, A. Zayed (eds.), Sampling, Approximation, and Signal Analysis. Applied and Numerical Harmonic Analysis. Birkh\"auser, Cham, 2023. 

\bibitem{Ches}
\textsc{J. Cheshmavar}
\newblock{Frame properties induced by iteration of a multiplication operator on Hardy spaces}
\url{https://arxiv.org/abs/2509.04879}

\bibitem{CH23}
\textsc{O. Christensen and M. Hasannasab},
\newblock {\em A survey on frame representations and operator orbits},
\newblock in: S. D. Casey, M. M. Dodson, P. J. S. G. Ferreira, A. Zayed (eds.), Sampling, Approximation, and Signal Analysis. Applied and Numerical Harmonic Analysis. Birkh\"auser, Cham, 2023.

\bibitem{CHP}
\textsc{O. Christensen, M. Hasannasab, and F. Philipp},
\newblock {\em Frame properties of operator orbits},
\newblock Math. Nachr. 293, 52--66 (2019).

\bibitem{DefGarMaeSev19}
\textsc{A. Defant, D. Garc\'{\i}a, M. Maestre, and P. Sevilla-Peris},
\newblock {\em Dirichlet Series and Holomorphic Functions in High Dimensions (New Mathematical Monographs)},
\newblock Cambridge University Press, 2019.

\bibitem{GMR}
\textsc{R. S. Garcia, J. Mashreghi, and W. T. Ross},
\newblock {\em Introduction to model spaces and their operators},
\newblock Cambridge Studies in Advanced Mathematics, 148. Cambridge University Press, Cambridge, 2016.

\bibitem{hedenmalm1997}
\textsc{H. Hedenmalm, P. Lindqvist, and K. Seip},
\newblock {\em A {H}ilbert space of {D}irichlet series and systems of dilated functions in ${L}_2(0,1)$},
\newblock Duke Math. J. 86(1), 1--37 (1997).

\bibitem{KasShu19}
\textsc{Z. A. Kasumov and A. S. Shukurov},
\newblock {\em On frame properties of iterates of a multiplication operator},
\newblock Result. Math. 74(2), Paper No. 84, 8 pp. (2019).

\bibitem{Mar15}
\textsc{M. Markovi\'c},
\newblock {\em Sharp inequalities over the unit polydisc},
\newblock J. Funct. Anal. 268, 2647--2671 (2015).

\bibitem{MAR}
\textsc{R. A. Mart\'inez-Avenda\~no and P. Rosenthal},
\newblock {\em An Introduction to Operators on the Hardy-Hilbert Space},
\newblock Springer New York, NY, 2006.

\bibitem{Nik2019}
\textsc{N. Nikolski},
\newblock {\em Hardy Spaces},
\newblock D. Gibbons and G. Gibbons (trans.), Cambridge University Press, 2019.

\bibitem{Paul}
\textsc{V. I. Paulsen and M. Raghupathi},
\newblock {\em An introduction to the theory of reproducing kernel Hilbert spaces},
\newblock Cambridge University Press.

\bibitem{QueQue20}
\textsc{H. Queff\'elec and M. Queff\'elec},
\newblock {\em Diophantine approximation and Dirichlet series},
\newblock Texts and Readings in Mathematics, vol. 80. Hindustan Book Agency, New Delhi; Springer, Singapore, 2nd edition, 2020.

\bibitem{Rud}
\textsc{W. Rudin},
\newblock {\em Function theory in polydiscs},
\newblock New York, NY, 1969.

\bibitem{Ste}
\textsc{M. Stessin},
\newblock {\em Wold decomposition of the Hardy space and Blaschke products similar to a contraction},
\newblock Colloq. Math. 81(2), 271--284 (1999).

\end{thebibliography}
\end{document}